\documentclass[a4paper,12pt,final]{amsart}
\usepackage{times,a4wide,mathrsfs,amssymb,amsmath,amsthm,enumerate,xypic,tikzsymbols,dsfont}

\newcommand{\C}{\mathbb{C}}
\newcommand{\ZZ}{\mathbb{Z}}

\newcommand{\QQ}{\mathbb{Q}}
\newcommand{\NN}{\mathbb{N}}
\newcommand{\PP}{\mathbb{P}}

\newcommand{\OO}{\mathcal O}

\newcommand{\YY}{\mathcal Y}

\newcommand{\JJ}{\mathcal J}

\newcommand{\MM}{\mathcal M}

\newcommand{\FF}{\mathcal F}

\newcommand{\rom}{\romannumeral}

\newcommand{\one}{\mathds{1}}

\DeclareMathOperator{\ide}{id}

\DeclareMathOperator{\ima}{Im}

\DeclareMathOperator{\sym}{Sym}
\DeclareMathOperator{\Gr}{Gr}

\DeclareMathOperator{\Dec}{Dec}

\newtheorem{theorem}{Theorem}[section]

\newtheorem{lemma}[theorem]{Lemma}

\newtheorem{corollary}[theorem]{Corollary}
\newtheorem{proposition}[theorem]{Proposition}

\newtheorem{convention}{Conventions}

\newtheorem{nonumbering}{Theorem}

\newtheorem{nonumberingc}{Corollary}

\newtheorem{nonumberingp}{Proposition}

\theoremstyle{definition}
\newtheorem{remark}[theorem]{Remark}
\newtheorem{definition}[theorem]{Definition}

\newtheorem{nonumberingt}{Acknowledgments}

\begin{document}

\author[Robert Laterveer]
{Robert Laterveer}

\address{Institut de Recherche Math\'ematique Avanc\'ee,
CNRS -- Universit\'e 
de Strasbourg,\
7 Rue Ren\'e Des\-car\-tes, 67084 Strasbourg CEDEX,
FRANCE.}
\email{robert.laterveer@math.unistra.fr}

\title{Algebraic cycles and Gushel--Mukai fivefolds}

\begin{abstract} We show that Gushel--Mukai fivefolds
 admit a multiplicative Chow--K\"unneth decomposition, in the sense of Shen--Vial. 
As a consequence, a certain tautological subring of the Chow ring of powers of these varieties injects into cohomology.
 \end{abstract}

\keywords{Algebraic cycles, Chow groups, motive, Bloch--Beilinson filtration, Beauville's ``splitting property'' conjecture, multiplicative Chow--K\"unneth decomposition, Fano variety, Gushel--Mukai variety}
\subjclass[2010]{Primary 14C15, 14C25, 14C30.}

\maketitle

\section{Introduction}

Given a smooth projective variety $Y$ over $\C$, let $A^i(Y):=CH^i(Y)_{\QQ}$ denote the Chow groups of $Y$ (i.e. the groups of codimension $i$ algebraic cycles on $Y$ with $\QQ$-coefficients, modulo rational equivalence). Intersection product defines a ring structure on $A^\ast(Y)=\bigoplus_i A^i(Y)$, the {\em Chow ring\/} of $Y$ \cite{F}. In the case of K3 surfaces, this ring structure has a peculiar property:

\begin{theorem}[Beauville--Voisin \cite{BV}]\label{K3} Let $S$ be a K3 surface. 
The $\QQ$-subalgebra
  \[  R^\ast(S):=  \bigl\langle  A^1(S), c_j(S) \bigr\rangle\ \ \ \subset\ A^\ast(S) \]
  injects into cohomology under the cycle class map.
  \end{theorem}

Inspired by the cases of K3 surfaces and abelian varieties, Beauville \cite{Beau3} has conjectured that for certain special varieties, the Chow ring should admit a {\em multiplicative splitting\/}. To make concrete sense of Beauville's elusive ``splitting property conjecture'', Shen--Vial \cite{SV} have introduced the concept of {\em multiplicative Chow--K\"unneth decomposition\/}; 
let us say ``MCK decomposition'' for brevity.

What can one say about the class of special varieties admitting an MCK decomposition ? This class is not yet well-understood. To give some idea: hyperelliptic curves have an MCK decomposition \cite[Example 8.16]{SV}, but the very general curve of genus $\ge 3$ does not have an MCK decomposition \cite[Example 2.3]{FLV2} (for more details, cf. subsection \ref{ss:mck} below). In the present note, we aim to enhance the understanding of this class by identifying some new members:

\begin{nonumbering}[=Theorem \ref{main}] Let $Y$ be a Gushel--Mukai fivefold. Then $Y$ has an MCK decomposition. 
\end{nonumbering}


Gushel--Mukai varieties come in two flavours: ordinary and special.
An ordinary Gushel--Mukai fivefold is a smooth dimensionally transverse intersection
  \[ Y=\Gr(2,5)\cap Q\ \ \ \subset\ \PP^9 \ \]
 of the Grassmannian $\Gr(2,5)$ (of 2-dimensional subspaces of a 5-dimensional vector space) and a quadric $Q$ (with respect to the Pl\"ucker embedding); these varieties have been studied at least since 1929 under the name of ``quadratic complexes of lines'' \cite{Seg}, \cite{Sem}, \cite{Roth}, \cite[Chapter 10]{Dol}.
For the definition of special Gushel--Mukai fivefolds, cf. subsection \ref{ss:gm} below. 

To prove Theorem \ref{main}, we have devised a general criterion (Proposition \ref{crit}), which we hope may be useful in other cases. As an illustration of this criterion, we show that certain Fano varieties of dimension 17 also have an MCK decomposition (Theorem \ref{main2}). In order to show that Gushel--Mukai fivefolds verify the criterion, we prove the following which may be of independent interest:

\begin{nonumberingp}[=Proposition \ref{van}] Let $Y$ be a Gushel--Mukai fivefold. The cycle class map
   \[ A^j_{}(Y)\ \to\ H^{2j}(Y,\QQ) \]
   is injective for all $j\not=3$.
   In particular, $Y$ has finite-dimensional motive (in the sense of \cite{Kim}).
\end{nonumberingp}

The existence of an MCK decomposition has remarkable intersection-theoretic consequences. This is exemplified by the following corollary, which is about a certain {\em tautological subring\/} of the Chow ring of powers of $Y$:

\begin{nonumberingc}[=Corollary \ref{cor1}]
Let $Y$ be a Gushel--Mukai fivefold, and $m\in\NN$. Let
  \[ R^\ast(Y^m):=\Bigl\langle (p_i)^\ast A^1(Y), \,  (p_i)^\ast A^2(Y), \, (p_{ij})^\ast(\Delta_Y)\Bigr\rangle\ \subset\ \ \ A^\ast(Y^m)   \]
  be the $\QQ$-subalgebra generated by (pullbacks of) divisors, codimension 2 cycles and the diagonal $\Delta_Y\in A^5(Y\times Y)$. (Here $p_i$ and $p_{ij}$ denote the various projections from $Y^m$ to $Y$ resp. to $Y\times Y$).
  The cycle class map induces injections
   \[ R^\ast(Y^m)\ \hookrightarrow\ H^\ast(Y^m,\QQ)\ \ \ \hbox{for\ all\ }m\in\NN\ .\]
   \end{nonumberingc}

That is, Gushel--Mukai fivefolds behave like K3 surfaces (cf. Theorem \ref{K3}) and hyperelliptic curves (cf. Remark \ref{tava} below), from the point of view of intersection theory.

It seems reasonable to expect that Gushel--Mukai threefolds (and Gushel--Mukai fourfolds and sixfolds) also have an MCK decomposition; establishing this is strictly more difficult than the dimension $5$ case (cf. Remark \ref{pity}).

 \vskip0.5cm

\begin{convention} In this note, the word {\sl variety\/} will refer to a reduced irreducible scheme of finite type over $\C$. A {\sl subvariety\/} is a (possibly reducible) reduced subscheme which is equidimensional. 

{\bf All Chow groups will be with rational coefficients}: we denote by $A_j(Y)$ the Chow group of $j$-dimensional cycles on $Y$ with $\QQ$-coefficients; for $Y$ smooth of dimension $n$ the notations $A_j(Y)$ and $A^{n-j}(Y)$ are used interchangeably. 
The notations $A^j_{\rm hom}(Y)$ and $A^j_{\rm AJ}(Y)$ will be used to indicate the subgroup of homologically trivial (resp. Abel--Jacobi trivial) cycles.

The contravariant category of Chow motives (i.e., pure motives with respect to rational equivalence as in \cite{Sc}, \cite{MNP}) will be denoted 
$\MM_{\rm rat}$.
\end{convention}

\section{Preliminaries}

\subsection{Gushel--Mukai fivefolds}
\label{ss:gm}

Gushel--Mukai varieties have been intensively studied \cite{D}, \cite{Per}, \cite{DK}, \cite{DK1}, \cite{DK3}, \cite{DK2}, \cite{KP}. In many senses, Gushel--Mukai varieties behave like cubic hypersurfaces; in particular, the theory of Gushel--Mukai fivefolds closely parallels that of cubic fivefolds (for instance, both have cohomology ``of curve type'').

\begin{definition}[Debarre--Kuznetsov \cite{DK}] A Gushel--Mukai variety is a variety $Y$ obtained as a smooth dimensionally transverse intersection
  \[ Y= \operatorname{CGr}(2,V_5)\cap \PP(W)\cap Q\ \ \ \ \subset\ \PP(\wedge^2 V_5\oplus \C)\cong \PP^{10}\ ,\]
  where $ \operatorname{CGr}(2,V_5)$ is the cone over the Grassmannian  $\operatorname{Gr}(2,V_5)$ (of $2$-dimensional subspaces in a fixed $5$-dimensional vector space $V_5$), and $\PP(W)$ and $Q$ are a linear subspace resp. a quadric.
  
  A Gushel--Mukai variety is called {\em special\/} if it contains the vertex of $ \operatorname{CGr}(2,V_5)$, and {\em ordinary\/} if it is not special.
  \end{definition}
  
  \begin{remark} For simplicity, we restrict to {\em smooth\/} Gushel--Mukai varieties (for the general set-up including singularities, cf. \cite{DK}). Gushel--Mukai varieties have dimension at most $6$. Examples of Gushel--Mukai varieties are: Clifford-general curves of genus $6$; Brill--Noether general polarized $K3$ surfaces of degree $10$ (i.e. genus $6$); smooth prime Fano threefolds of degree $10$ (i.e. genus $6$). 
  
 It is possible to give an intrinsic characterization of Gushel--Mukai varieties \cite[Theorem 2.3]{DK}.
 
 Projecting from the vertex of the cone, one finds that an ordinary Gushel--Mukai variety is isomorphic to a smooth dimensionally transverse intersection
   \[ Y=\Gr(2,V_5)\cap \PP(W^\prime)\cap Q^\prime\ \ \ \subset\ \PP(\wedge^2 V_5)\cong\PP^9\ ,\]
   with $W^\prime\subset \wedge^2 V_5$ a linear subspace (isomorphic to $W$) and $Q^\prime$ a quadric. A special Gushel--Mukai variety is a double cover of a linear section of the Grassmannian.
  \end{remark}

  In this note, we will be interested in Gushel--Mukai fivefolds. 

\begin{proposition}\label{dm} Let $Y$ be a Gushel--Mukai fivefold.

\noindent
(0) $Y$ is a Fano variety with Picard number $1$, index $3$ and degree 10 (and conversely, every smooth Fano fivefold with Picard number $1$, index $3$ and degree 10
is a Gushel--Mukai fivefold).

\noindent
(\rom1) $Y$ is rational.

\noindent
(\rom2)
The Hodge diamond of $Y$ is
    \[ \begin{array}[c]{ccccccccccc}
      &&&&& 1 &&&&&\\
      &&&&0&&0&&&&\\
      &&&0&&1&&0&&&\\
      &&0&&0&&0&&0&&\\
      &0&&0&&2&&0&&0&\\
      0&&0&&10&&10&&0&&0\\
            &0&&0&&2&&0&&0&\\     
      &&0&&0&&0&&0&&\\
       &&&0&&1&&0&&&\\
        &&&&0&&0&&&&\\      
        &&&&& 1 &&&&&\\
\end{array}\]
\end{proposition}

\begin{proof} Point (0) is attributed to Mukai \cite[Theorem 1.1]{D}.

\noindent
(\rom1) This is \cite[p. 96]{Roth} or \cite[Proposition 4.2]{DK}. 

\noindent
(\rom2) This is \cite[Proposition 3.1]{DK1} (cf. also \cite[Introduction]{Nag} and \cite[Section 3.1]{FM0}).
\end{proof}

Just as Gushel--Mukai fourfolds (and sixfolds) are related to double EPW sextics, Gushel--Mukai fivefolds (and threefolds) are related to ``double EPW surfaces'' (a double EPW surface is a certain double cover of the singular locus of an EPW sextic, cf. \cite[Theorem 5.2]{DK4} and \cite{DK2}).

\begin{theorem}[Debarre--Kuznetsov \cite{DK2}]\label{pp} Let $Y$ be a general Gushel--Mukai fivefold. There exists a double EPW surface $S$ and a canonical (correspondence-induced) isomorphism
  \[  H_5(Y,\ZZ)\ \xrightarrow{\cong}\ H_1(S,\ZZ)\ .\]
  This induces an isomorphism of principally polarized abelian varieties
    \[ \JJ^3(Y)\ \xrightarrow{\cong}\ \operatorname{Alb}(S) \]
    from the intermediate jacobian of $Y$ to the Albanese variety of $S$.
\end{theorem}

\begin{proof} This is \cite[Theorem 1.1]{DK2}. In this note, we do not need the full force of Theorem \ref{pp}; we only need the following, which is a key ingredient in the proof of loc. cit.:

\begin{lemma}[Debarre--Kuznetsov \cite{DK2}]\label{ldk} Let $Y$ be a general Gushel--Mukai fivefold, and let $F^2_\sigma(Y)$ be the smooth curve parametrizing $\sigma$-planes contained in $Y$. Then the universal $\sigma$-plane $P\subset F^2_\sigma(Y)\times Y$ determines a surjection
  \[ P_\ast\colon\ \   H_1(F^2_\sigma(Y),\ZZ)\ \twoheadrightarrow\ H_5(Y,\ZZ)\ .\]
  \end{lemma}

(This is \cite[Lemma 5.8]{DK2}.)
\end{proof}

\subsection{MCK decomposition}
\label{ss:mck}

\begin{definition}[Murre \cite{Mur}] Let $X$ be a smooth projective variety of dimension $n$. We say that $X$ has a {\em CK decomposition\/} if there exists a decomposition of the diagonal
   \[ \Delta_X= \pi^0_X+ \pi^1_X+\cdots +\pi_X^{2n}\ \ \ \hbox{in}\ A^n(X\times X)\ ,\]
  such that the $\pi^i_X$ are mutually orthogonal idempotents and $(\pi_X^i)_\ast H^\ast(X,\QQ)= H^i(X,\QQ)$.
  
  (NB: ``CK decomposition'' is shorthand for ``Chow--K\"unneth decomposition''.)
\end{definition}

\begin{remark} The existence of a CK decomposition for any smooth projective variety is part of Murre's conjectures \cite{Mur}, \cite{J4}. 
\end{remark}

\begin{definition}[Shen--Vial \cite{SV}] Let $X$ be a smooth projective variety of dimension $n$. Let $\Delta_X^{\rm sm}\in A^{2n}(X\times X\times X)$ be the class of the small diagonal
  \[ \Delta_X^{\rm sm}:=\bigl\{ (x,x,x)\ \vert\ x\in X\bigr\}\ \subset\ X\times X\times X\ .\]
  An {\em MCK decomposition\/} is a CK decomposition $\{\pi_X^i\}$ of $X$ that is {\em multiplicative\/}, i.e. it satisfies
  \[ \pi_X^k\circ \Delta_X^{\rm sm}\circ (\pi_X^i\times \pi_X^j)=0\ \ \ \hbox{in}\ A^{2n}(X\times X\times X)\ \ \ \hbox{for\ all\ }i+j\not=k\ .\]
  
 (NB: ``MCK decomposition'' is shorthand for ``multiplicative Chow--K\"unneth decomposition''.) 
  \end{definition}
  
  \begin{remark} The small diagonal (seen as a correspondence from $X\times X$ to $X$) induces the {\em multiplication morphism\/}
    \[ \Delta_X^{\rm sm}\colon\ \  h(X)\otimes h(X)\ \to\ h(X)\ \ \ \hbox{in}\ \MM_{\rm rat}\ .\]
 Let us assume $X$ has a CK decomposition
  \[ h(X)=\bigoplus_{i=0}^{2n} h^i(X)\ \ \ \hbox{in}\ \MM_{\rm rat}\ .\]
  By definition, this decomposition is multiplicative if for any $i,j$ the composition
  \[ h^i(X)\otimes h^j(X)\ \to\ h(X)\otimes h(X)\ \xrightarrow{\Delta_X^{\rm sm}}\ h(X)\ \ \ \hbox{in}\ \MM_{\rm rat}\]
  factors through $h^{i+j}(X)$.
  
  If $X$ has an MCK decomposition, then setting
    \[ A^i_{(j)}(X):= (\pi_X^{2i-j})_\ast A^i(X) \ ,\]
    one obtains a bigraded ring structure on the Chow ring: that is, the intersection product sends $A^i_{(j)}(X)\otimes A^{i^\prime}_{(j^\prime)}(X) $ to  $A^{i+i^\prime}_{(j+j^\prime)}(X)$.
    
   It is expected that for any $X$ with an MCK decomposition, one has
    \[ A^i_{(j)}(X)\stackrel{??}{=}0\ \ \ \hbox{for}\ j<0\ ,\ \ \ A^i_{(0)}(X)\cap A^i_{\rm hom}(X)\stackrel{??}{=}0\ ;\]
    this is related to Murre's conjectures B and D, that have been formulated for any CK decomposition \cite{Mur}.

  The property of having an MCK decomposition is restrictive, and is closely related to Beauville's ``splitting property conjecture'' \cite{Beau3}. 
  To give an idea: hyperelliptic curves have an MCK decomposition \cite[Example 8.16]{SV}, but the very general curve of genus $\ge 3$ does not have an MCK decomposition \cite[Example 2.3]{FLV2}. As for surfaces: a smooth quartic in $\PP^3$ has an MCK decomposition, but a very general surface of degree $ \ge 7$ in $\PP^3$ should not have an MCK decomposition \cite[Proposition 3.4]{FLV2}.
For more detailed discussion, and examples of varieties with an MCK decomposition, we refer to \cite[Section 8]{SV}, as well as \cite{V6}, \cite{SV2}, \cite{FTV}, \cite{37}, \cite{38}, \cite{39}, \cite{40}, \cite{FLV2}.
   \end{remark}

\subsection{Trivial Chow groups} We recall a folklore result:

\begin{lemma}\label{triv} Let $M$ be a smooth projective variety. The following are equivalent:

\noindent
(\rom1) The motive of $M$ is of Tate type: 
    \[ h(M)\cong \bigoplus \one(\ast)\ \ \ \hbox{in}\ \MM_{\rm rat}\ ;\]

\noindent
(\rom2) The cycle class map induces an isomorphism $A^\ast(M)\cong H^\ast(M,\QQ)$\,;

\noindent
(\rom3) $A^\ast_{\rm hom}(M)=0$\,;

\noindent
(\rom4) $A^\ast(M)$ is a finite-dimensional $\QQ$-vector space;

\noindent
(\rom5) The natural map $A^\ast(M)\otimes A^\ast(M)\to A^\ast(M\times M)$ is an isomorphism.
\end{lemma}

\begin{proof} The implications (\rom1)$\Rightarrow$(\rom2)$\Rightarrow$(\rom3)$\Rightarrow$(\rom4) are obvious. The implication (\rom4)$\Rightarrow$(\rom1) is \cite{Kim2} or \cite{V10}. The implication (\rom1)$\Rightarrow$(\rom5) follows readily from the fact that $h(M\times M)=h(M)\otimes h(M)$ and $\one(\ell)\otimes\one(m)=\one(\ell+m)$. Finally, to see that (\rom5)$\Rightarrow$(\rom4), one notes that (\rom5) implies the decomposition of the diagonal
  \[ \Delta_M=\sum_{j=1}^r \alpha_j\times \beta_j\ \ \ \hbox{in}\ A^{\dim M}(M\times M)\ ,\]
  where $\alpha_j,\beta_j\in A^\ast(M)$.
 Letting this decomposition act on $A^\ast(M)$, one finds that the identity factors over an $r$-dimensional $\QQ$-vector space, and so (\rom4) holds.
 \end{proof}
 
 \begin{definition}[Voisin {\cite[Section 3.1]{V0}}] A smooth projective variety $M$ is said to have {\em trivial Chow groups\/} if $M$ verifies any of the equivalent conditions of Lemma \ref{triv}.
 \end{definition}

 \subsection{The Franchetta property}
 \label{ss:fr}

 \begin{definition} Let $\YY\to B$ be a smooth projective morphism, where $\YY, B$ are smooth quasi-projective varieties. We say that $\YY\to B$ has the {\em Franchetta property in codimension $j$\/} if the following holds: for every $\Gamma\in A^j(\YY)$ such that the restriction $\Gamma\vert_{Y_b}$ is homologically trivial for the very general $b\in B$, the restriction $\Gamma\vert_b$ is zero in $A^j(Y_b)$ for all $b\in B$.
 
 We say that $\YY\to B$ has the {\em Franchetta property\/} if $\YY\to B$ has the Franchetta property in codimension $j$ for all $j$.
 \end{definition}
 
 This property is studied in \cite{PSY}, \cite{BL}, \cite{FLV}, \cite{FLV3}.
 
 \begin{definition} Given a family $\YY\to B$ as above, with $Y:=Y_b$ a fiber, we write
   \[ GDA^j_B(Y):=\ima\Bigl( A^j(\YY)\to A^j(Y)\Bigr) \]
   for the subgroup of {\em generically defined cycles}. 
  In a context where it is clear to which family we are referring, the index $B$ will often be suppressed from the notation.
  \end{definition}
  
  With this notation, the Franchetta property amounts to saying that $GDA^\ast_B(Y)$ injects into cohomology, under the cycle class map.

   \subsection{A Franchetta-type result}

  \begin{proposition}\label{Frtype} Let $M$ be a smooth projective variety with trivial Chow groups. Let $L_1,\ldots,L_r\to {M}$ be very ample line bundles, and let
  $\YY\to B$ be the universal family of smooth dimensionally transverse complete intersections of type 
    \[ Y={M}\cap H_1\cap\cdots\cap H_r\ ,\ \ \  H_j\in\vert L_j\vert\ .\]
  Assume the fibers $Y=Y_b$ have $H^{\dim Y}_{\rm tr}(Y,\QQ)\not=0$.
  There is an inclusion
    \[ \ker \Bigl( GDA^{\dim Y}_B(Y\times Y)\to H^{2\dim Y}(Y\times Y,\QQ)\Bigr)\ \ \subset\ \Bigl\langle (p_1)^\ast GDA^\ast_B(Y), (p_2)^\ast GDA^\ast_B(Y)  \Bigr\rangle\ .\]
   \end{proposition}
   
   \begin{proof} This is essentially Voisin's ``spread'' result \cite[Proposition 1.6]{V1} (cf. also \cite[Proposition 5.1]{LNP} for a reformulation of Voisin's result). We give a proof which is somewhat different from \cite{V1}. Let $\bar{B}:=\PP H^0({M},L_1\oplus\cdots\oplus L_r)$ (so $B\subset \bar{B}$ is a Zariski open), and let us consider the projection
   \[ \pi\colon\ \  \YY\times_{\bar{B}} \YY\ \to\ M\times M\ .\]
   Using the very ampleness assumption, one finds that $\pi$ is a $\PP^s$-bundle over $({M}\times {M})\setminus \Delta_{{M}}$, and a $\PP^t$-bundle over $\Delta_{{M}}$.
   That is, $\pi$ is what is termed a {\em stratified projective bundle\/} in \cite{FLV}. As such, \cite[Proposition 5.2]{FLV} implies the equality
      \begin{equation}\label{stra} GDA^\ast_B(Y\times Y)= \ima\Bigl( A^\ast(M\times M)\to A^\ast(Y\times Y)\Bigr) +  \Delta_\ast GDA^\ast_B(Y)\ ,\end{equation}
      where $\Delta\colon Y\to Y\times Y$ is the inclusion along the diagonal. As $M$ has trivial Chow groups, $A^\ast(M\times M)$ is generated by $A^\ast(M)\otimes A^\ast(M)$ (Lemma \ref{triv}). 
      Base-point freeness of the $L_j$ implies that 
        \[  GDA^\ast_B(Y)=\ima\bigl( A^\ast(M)\to A^\ast(Y)\bigr)\ .\]
       The equality \eqref{stra} thus reduces to
      \[ GDA^\ast_B(Y\times Y)=\Bigl\langle (p_1)^\ast GDA^\ast_B(Y), (p_2)^\ast GDA^\ast_B(Y), \Delta_Y\Bigr\rangle\ \]   
      (where $p_1, p_2$ denote the projection from $S\times S$ to first resp. second factor). The assumption that $Y$ has non-zero transcendental cohomology
      implies that the class of $\Delta_Y$ is not decomposable in cohomology. It follows that
      \[ \begin{split}  \ima \Bigl( GDA^{\dim Y}_B(Y\times Y)\to H^{2\dim Y}(Y\times Y,\QQ)\Bigr) =&\\
       \ima\Bigl(  \Dec^{\dim Y}(Y\times Y)\to H^{2\dim Y}(Y\times Y,\QQ)\Bigr)& \oplus \QQ[\Delta_Y]\ ,\\
       \end{split}\]
      where we use the shorthand 
       \[ \Dec^j(Y\times Y):= \Bigl\langle (p_1)^\ast GDA^\ast_B(Y), (p_2)^\ast GDA^\ast_B(Y)\Bigr\rangle\cap A^j(Y\times Y) \ \]     
       for the {\em decomposable cycles\/}. 
       We now see that if $\Gamma\in GDA^{\dim Y}(Y\times Y)$ is homologically trivial, then $\Gamma$ does not involve the diagonal and so $\Gamma\in \Dec^{\dim Y}(Y\times Y)$.
       This proves the proposition.
         \end{proof}
  
  \begin{remark} Proposition \ref{Frtype} has the following consequence: if the family $\YY\to B$ has the Franchetta property, then $\YY\times_B \YY\to B$ has the Franchetta property in codimension $\dim Y$.
   \end{remark}

 \subsection{A CK decomposition}

\begin{lemma}\label{ck} Let $M$ be a smooth projective variety with trivial Chow groups. Let $Y\subset M$ be a smooth complete intersection of dimension $\dim Y=d$ defined by ample line bundles.
 The variety $Y$ has a self-dual CK decomposition $\{\pi^j_Y\}$ with the property that
   \[  h^j(Y):=(Y,\pi^j_Y,0) =\oplus \one(\ast)\ \ \ \hbox{in}\ \MM_{\rm rat}\ \ \ \forall\ j\not=d \ .\]
   
   Moreover, this CK decomposition is {\em generically defined\/}: writing $\YY\to B$ for the universal family (of complete intersections of the type of $Y$),
 there exist relative projectors $\pi^j_\YY\in A^{d}(\YY\times_B \YY)$ such that $\pi^j_Y=\pi^j_\YY\vert_{b}$ (where $Y=Y_b$ for $b\in B$). 
     \end{lemma} 
 
 \begin{proof} This is a standard construction, one can look for instance at \cite{Pet} (in case $d$ is odd, which will be the case in this note,  the ``variable motive'' $h(Y)^{\rm var}$ of 
 \cite[Theorem 4.4]{Pet}  coincides with $h^{d}(Y)$).
\end{proof}

\section{Main results}

\subsection{Vanishing of $A_1^{\rm hom}$}

\begin{proposition}\label{van} Let $Y$ be a Gushel--Mukai fivefold. Then
  \[ A^i_{\rm hom}(Y)=0\ \ \ \forall\ i\not= 3\ .\]
  In particular, $Y$ has finite-dimensional motive (in the sense of \cite{Kim}, \cite{An}).
\end{proposition}

\begin{proof} It is a basic fact that any smooth projective variety $M$ with $A_0^{\rm hom}(M)=A_1^{\rm hom}(M)=0$ also has $A^2_{\rm AJ}(M)=A^3_{\rm AJ}(M)=0$. (This ``basic fact''
can be proven using the Bloch--Srinivas ``decomposition of the diagonal'' argument, cf. \cite[Remark 1.8.1]{4}.) Since $H^3(Y,\QQ)=0$, the inclusion $A^2_{\rm AJ}(Y)\subset A^2_{\rm hom}(Y)$ is an equality.

Note that $Y$ is Fano, hence rationally connected and so $A_0(Y)\cong\QQ$. Thus, to prove the proposition it only remains to prove that $A_1^{\rm hom}(Y)=0$.
To prove this vanishing, we reason {\em family-wise\/}, and exploit instances of the {\em Franchetta property\/}.

In a first step, let us treat the case of {\em ordinary\/} Gushel--Mukai fivefolds. That is, we write 
  \[ B\ \subset\ \bar{B}:=\PP H^0\bigl(\Gr(2,5),\OO_{\Gr(2,5)}(2)\bigr) \cong\PP^r\ ,\]
  where $B$ is the Zariski open parametrizing smooth dimensionally transverse hypersurfaces $Y_b\subset \Gr(2,5)$. There is a universal family
  \[ \bar{\YY}:= \Bigl\{ (g,b)\ \in\ \Gr(2,5)\times \bar{B}\ \Big\vert\ g\in Y_b\Bigr\}\ \ \ \subset\ \Gr(2,5)\times \bar{B}\ , \]
 and a universal family $\YY\to B$ of smooth hypersurfaces.
  

We are in the set-up of Proposition \ref{Frtype} (with $M$ being the Grassmannian $\Gr(2,5)$), and so Proposition \ref{Frtype} gives us an inclusion
   \begin{equation}\label{equa} \ker \Bigl( GDA^{5}_B(Y\times Y)\to H^{10}(Y\times Y,\QQ)\Bigr)\ \ \subset\ \Bigl\langle (p_1)^\ast GDA^\ast_B(Y), (p_2)^\ast GDA^\ast_B(Y)  \Bigr\rangle\ .\end{equation}
   
Let us construct an interesting cycle in $GDA^{5}_B(Y\times Y)$ to which we can apply \eqref{equa}. For general $Y=Y_b$, Lemma \ref{ldk} gives us a smooth curve $F:=F^2_\sigma(Y)$ and a subvariety $P\subset F\times Y$ inducing a surjection
  \[ P_\ast\colon\ \     H^1(F,\QQ)\ \twoheadrightarrow\ H^5(Y,\QQ)\ .\]
  Writing $\FF\to B$ for the universal family of varieties of $\sigma$-planes, the subvariety $P$ naturally exists relatively, i.e. $P\in GDA^3(F\times Y)$. 
Since both $F$ and $Y$ verify the standard conjectures, the right-inverse to $P_\ast$ is correspondence-induced, i.e. there exists $Q\in A^{3}(Y\times F)$ such that
  \[ (P\circ Q)_\ast =\ide\colon\ \ H^5(Y,\QQ)\ \to\ H^5(Y,\QQ) \ \]
  (This follows as in \cite[Proof of Proposition 1.1]{V4}).\footnote{It seems likely that one can take $Q$ to be a multiple of the transpose of $P$. For this, one would need to know that the map $P_\ast$ of Lemma \ref{ldk} is compatible with cup-product up to a multiple; this is the point of view taken in \cite{V1} to create the inverse correspondence $Q$.}
  
  We now involve the (generically defined) CK decomposition $\pi^j_Y\in A^5(Y\times Y)$ given by Lemma \ref{ck}. The above means that for $Y=Y_b$ with $b\in B$ sufficiently general, there is vanishing
    \[  (\Delta_Y -P\circ Q)\circ \pi^5_Y=0\ \ \ \hbox{in}\ H^{10}(Y\times Y,\QQ)\ .\]
  
  Applying Voisin's Hilbert schemes argument \cite[Proposition 3.7]{V0}, \cite[Proposition 4.25]{Vo} (cf. also \cite[Proposition 2.10]{Lfam} for the precise form used here), we can assume that $Q$ is also generically defined, and hence
    \[  (\Delta_Y -P\circ Q)\circ \pi^5_Y\ \ \in\ GDA^5(Y\times Y)\ .\]
  Now we learn from \eqref{equa} that this cycle is decomposable, i.e,
   \[      (\Delta_Y -P\circ Q)\circ \pi^5_Y\ \ \in\ \Bigl\langle (p_1)^\ast GDA^\ast(Y), (p_2)^\ast GDA^\ast(Y)  \Bigr\rangle\ . \]
   This is true for $Y$ sufficiently general, but then thanks to the usual spread lemma \cite[Lemma 3.2]{Vo}, this extends over all of $B$. That is, for any $Y=Y_b$ with $b\in B$ we can write
   \[   (\Delta_Y -P\circ Q)\circ \pi^5_Y=\gamma\ \ \ \hbox{in}\ A^5(Y\times Y)\ ,\]   
   with $\gamma\in A^\ast(Y)\otimes A^\ast(Y)$. Since the $\pi^j_Y, j\not=5$ of Lemma \ref{ck} are decomposable (i.e. they are in $A^\ast(Y)\otimes A^\ast(Y)$), 
   this implies that we can write
   \[   \Delta_Y -P\circ Q=\gamma^\prime\ \ \ \hbox{in}\ A^5(Y\times Y)\ ,\]      
   with $\gamma^\prime\in A^\ast(Y)\otimes A^\ast(Y)$. Being decomposable, $\gamma^\prime$ does not act on Abel--Jacobi trivial cycles, and so
   \[ A^i_{\rm AJ}(Y)\ \xrightarrow{Q_\ast}\ A^{i-2}_{\rm AJ}(F)\ \xrightarrow{P_\ast}\ A^i_{\rm AJ}(Y) \]
   is the identity. But $F$ being a curve, the group in the middle vanishes for all $i$. This proves the vanishing
     \[   A^\ast_{\rm AJ}(Y)=0 \] 
 for ordinary Gushel--Mukai fivefolds $Y$. 
 The Kimura finite-dimensionality of $Y$ then follows from \cite[Theorem 4]{43}.
 
 In the second step of the proof, let us extend to {\em all\/} (i.e. both ordinary and special) Gushel--Mukai fivefolds. To this end, let 
 us write $\YY^\prime\to B^\prime$ for the universal family of all Gushel--Mukai fivefolds, where
  \[ B^\prime\ \subset\ \bar{B}^\prime:= \PP H^0(  \operatorname{CGr}(2,5) ,\OO_{}(1)\oplus \OO_{}(2)  )  \]
  is the dense Zariski open parametrizing smooth dimensionally transverse intersections. There exists a dense Zariski open $B^\prime_{\rm ord}\subset B^\prime$ parametrizing the ordinary Gushel--Mukai fivefolds (this is the locus where the section misses the summit of the cone $ \operatorname{CGr}(2,V_5)$). This means that there is a morphism $B^\prime_{\rm ord}\to B_{\rm ord}$ where
   \[ B_{\rm ord}\ \subset\ \bar{B}:= \PP H^0( \Gr(2,5)   ,\OO(2)  )  \]
  is the dense Zariski open parametrizing smooth dimensionally transverse intersections, and there is a base change diagram
  \[   \begin{array}[c]{ccc}   \YY^\prime & \to & \YY\\
                                             &&\\
                                             \downarrow&&\downarrow\\
                                             B^\prime_{\rm ord}&\to& \ \ B_{\rm ord} \ . \\
                                             \end{array}\]
                                             
   Thanks to the above first step, the relative projector $\pi^5_Y$ factors over the motive of the curve $F$, for each $Y=Y_b$ with $b\in B_{\rm ord}$.
 Doing the base change $B^\prime_{\rm ord}\to B_{\rm ord}$, and extending to all of $B^\prime$ using the spread lemma \cite[Lemma 3.2]{Vo}, the same is then true for all $b\in B^\prime$, i.e.
 for any Gushel--Mukai fivefold there is a split injection
   \[ h(Y)\ \hookrightarrow\ h(F)(-2)\oplus \bigoplus\one(\ast)\ \ \ \hbox{in}\ \MM_{\rm rat}\ ,\]
   where $F$ is a curve.
   This implies the vanishing of $A^4_{\rm hom}(Y)$, and also the vanishing $A^\ast_{\rm AJ}(Y)=0$ for any Gushel--Mukai fivefold $Y$. The Kimura finite-dimensionality of $Y$ then follows from the forementioned \cite[Theorem 4]{43} (or directly from the fact that $h(Y)$ is a submotive of a sum of twists of motives of curves).
\end{proof}

\begin{remark} Combining Proposition \ref{van} and the Bloch--Srinivas argument \cite{BS}, one finds that $H^5(Y,\QQ)$ is supported in codimension 2, i.e. the generalized Hodge conjecture is true for {\em any\/} Gushel--Mukai fivefold. This improves on \cite[Corollary 3.8]{Nag}, where the generalized Hodge conjecture was proven for the {\em general\/} ordinary Gushel--Mukai fivefold.
\end{remark}

\subsection{MCK for GM fivefolds}

In order to prove our main result (Theorem \ref{main}), we give a general criterion for having an MCK. The criterion applies to varieties $Y$ of odd dimension $\ge 5$ that are ``of curve-type'', i.e. $A^\ast_{\rm AJ}(Y)=0$.

\begin{proposition}\label{crit} Let $V$ be a smooth projective variety of dimension $2m\ge 6$, and let $Y\subset V$ be a smooth hypersurface defined by a very ample divisor $h=\OO_V(1)$ on $V$.
Assume the following conditions hold:

\noindent
(c1) $V$ has trivial Chow groups (i.e. $A^\ast_{\rm hom}(V)=0$) and $A^1(V)=\QQ[h]$;

\noindent
(c2) \[  A_i^{\rm hom}(Y)=0\ \ \ \forall\ i\le m-2\ ;\]

\noindent
(c3) The cycle class map induces an injection
  \[ \ima\Bigl( A^m(V)\to A^m(Y)\Bigr) \ \hookrightarrow\ H^{2m}(Y,\QQ)\ .\]

Then $Y$ has an MCK decomposition.
\end{proposition}

\begin{proof} As so often, it helps to reason {\em family-wise\/}, and exploit instances of the {\em Franchetta property\/}. That is, we write 
  \[ B\ \subset\ \bar{B}:=\PP H^0\bigl(V,\OO_V(1)\bigr) \cong\PP^r\ ,\]
  where $B$ is the Zariski open parametrizing smooth dimensionally transverse hypersurfaces $Y_b\subset V$. There is a universal family
  \[ \bar{\YY}:= \Bigl\{ (v,b)\ \in\ V\times \bar{B}\ \Big\vert\ v\in Y_b\Bigr\}\ \ \ \subset\ V\times \bar{B}\ , \]
 and a universal family $\YY\to B$ of smooth hypersurfaces.
  
  The base-point-freeness of $\OO_V(1)$ ensures that each point in $V$ imposes one condition on $\bar{B}$, and so
    \[  \bar{\YY}\ \to\ V \]
  is a $\PP^{r-1}$-bundle. This readily gives (cf. for instance \cite[Proposition 5.2]{FLV}) the equality
   \[  GDA^\ast(Y):=GDA^\ast_B(Y)  =\ima\Bigl( A^\ast(V)\to A^\ast(Y)\Bigr)  \ .\]
    
Condition (c2), plus the Bloch--Srinivas argument \cite{BS}, ensures that there exists a curve $C$ and a split injection of motives
  \begin{equation}\label{injmot}  h(Y)\ \hookrightarrow\ h(C)(1-m) \oplus \bigoplus\one(\ast)\ \ \ \hbox{in}\ \MM_{\rm rat}\ ,\end{equation}
and in particular
$Y=Y_b$ has only one non-trivial Chow group:
  \[ A^i_{\rm hom}(Y)=0\ \ \ \forall\ i\not=m\ .\]
  Condition (c3) ensures that $GDA^m(Y)$ injects into cohomology under the cycle class map, and so $\YY\to B$ has the Franchetta property.
  
  Let us now turn to $\YY\times_B \YY$. The very ampleness of $\OO_V(1)$ ensures that 2 different points in $V$ impose two independent conditions on $\bar{B}$, and so
    \[  \bar{\YY}\times_{\bar{B}}\ \bar{\YY} \ \to\ V\times V \]
  is a $\PP^{r-2}$-bundle over $(V\times V)\setminus \Delta_V$, and a $\PP^{r-1}$-bundle over the diagonal $\Delta_V$. This readily gives (for instance, by applying \cite[Proposition 5.2]{FLV}) the equality
   \[ \begin{split}  GDA^\ast(Y\times Y):=GDA^\ast_B(Y\times Y)  =\ima\Bigl( A^\ast(V\times V)\to A^\ast(Y\times Y)\Bigr) +&\\    \Delta_\ast \ima   \Bigl( A^\ast(V)\to &A^\ast(Y)\Bigr)   \ ,\\
   \end{split}\]
   where $\Delta\colon Y\hookrightarrow Y\times Y$ denotes the diagonal embedding.
  The assumption that $V$ has trivial Chow groups implies that the product map induces an isomorphism
    \[ A^\ast(V)\otimes A^\ast(V)\ \xrightarrow{\cong}\   A^\ast(V\times V)\ .\]
  Moreover, pushing forward along $\Delta$ is the same as pulling back under one of the projections $p_j\colon Y\times Y\to Y$ and intersecting with $\Delta_Y$.
  
  The above equality thus boils down to
   \[  GDA^\ast(Y\times Y) = \Bigl\langle  (p_1)^\ast GDA^\ast(Y), (p_2)^\ast GDA^\ast(Y), \Delta_Y\Bigr\rangle\ \]
   (where we write $\langle - , -\rangle$ for the $\QQ$-algebra generated by the elements in brackets). Let us now verify that the restriction of the cycle class map 
   \begin{equation}\label{want} GDA^i(Y\times Y)\ \to\ H^{2i}(Y\times Y,\QQ)\  \end{equation}  
   is injective for all $i$
 (in other words, that $\YY\times_B \YY\to B$ has the Franchetta property). Since the split injection \eqref{injmot} is generically defined, it induces an injection
   \[ GDA^i(Y\times Y)\ \hookrightarrow\ A^{i+2-2m}(C\times C)\oplus \bigoplus GDA^\ast(Y)\oplus \QQ^s\ .\]
   In particular, for $i>2m$ and for $i<2m-1$ the injectivity of \eqref{want} follows from the Franchetta property for $\YY\to B$. 
   
   It only remains to check injectivity of \eqref{want} for $i=2m$ and for $i=2m-1$:
   
   For $i=2m$, we observe that
   \begin{equation}\label{exc}  \Delta_Y\cdot (p_j)^\ast(h) \ \ \in\  \Bigl\langle  (p_1)^\ast GDA^\ast(Y), (p_2)^\ast GDA^\ast(Y)\Bigr\rangle   \ ,\end{equation}
   and so the injectivity of \eqref{want} for $i=2m$ again reduces to the Franchetta property for $\YY\to B$. The fact \eqref{exc} follows from the excess intersection formula, or alternatively one can reason as follows:  let $\tau\colon Y\hookrightarrow V$ denote the inclusion morphism. There is equality
   \[    \Delta_Y\cdot (p_j)^\ast(h) =  {}^t \Gamma_\tau\circ \Gamma_\tau
                                                                        = (\tau\times\tau)^\ast(\Delta_V)   \ \ \ \hbox{in}\ A^{2m}(Y\times Y)\ \]
                                                                        (where the second equality follows from Lieberman's lemma). But $V$ having trivial Chow groups, the diagonal $\Delta_V$ is completely decomposed, i.e. $\Delta_V\in A^\ast(V)\otimes A^\ast(V)$. This proves \eqref{exc}.
                                                                        
Finally for $i=2m-1$, let us note that without loss of generality we may suppose that $H^{2m-1}(Y,\QQ)\not=0$ (indeed, if $H^{2m-1}(Y,\QQ)=0$ then $A^\ast_{\rm hom}(Y)=0$ and so $Y$ has an MCK decomposition for trivial reasons). This means that the class of $\Delta_Y$ in cohomology is linearly independent from the decomposable cycles $\langle  (p_1)^\ast GDA^\ast(Y), (p_2)^\ast GDA^\ast(Y)\rangle$, i.e. 
  \[  \begin{split}  &\ima \Bigl(      GDA^{2m-1}(Y\times Y)\to H^{4m-2}(Y\times Y,\QQ)\Bigr)=\\     &\ima \Bigl( \Bigl\langle  (p_1)^\ast GDA^\ast(Y), (p_2)^\ast GDA^\ast(Y)\Bigr\rangle\to H^{4m-2}(Y\times Y,\QQ)\Bigr) + \QQ[\Delta_Y]\ .\\
  \end{split}\]
  The injectivity of \eqref{want} for $i=2m-1$ thus reduces to the Franchetta property for $\YY\to B$. We have now proven that $\YY\times_B \YY\to B$ has the Franchetta property.
  


 To conclude, let us now establish that the CK decomposition of Lemma \ref{ck} is MCK. By definition, what we need to check is that the cycle
   \[  \Gamma_{ijk}:=   \pi_Y^k\circ \Delta_Y^{\rm sm}\circ (\pi_Y^i\times \pi_Y^j)\ \ \ \in\ A^{4m-2}(Y\times Y\times Y) \]
   is zero for all $i+j\not=k$.
   
  Let us assume at least one of the integers $i,j,k$ is even. In this case, there is an injection
    \[ \Gamma_{ijk}\ \ \in \ ( \pi^{2m-1-i}_Y\times \pi^{2m-1-j}_Y \times \pi^k_Y)_\ast  A^{4m-2}(Y\times Y\times Y)\ \hookrightarrow\ \bigoplus A^\ast(Y\times Y)\ ,\]
    and this injection sends generically defined cycles to generically defined cycles. But $\Gamma_{ijk}$ is generically defined and homologically trivial, and so the Franchetta property for $\YY\times_B \YY\to B$ gives the required vanishing  $\Gamma_{ijk}=0$.
    
   Next, let us assume $i=j=k=2m-1$. In this case, the injection of motives 
     \[ h^{2m-1}(Y)\ \hookrightarrow\  h^1(C)(1-m)\] 
  induces an injection of Chow groups
   \[ \Gamma_{ijk}\ \ \in\ ( \pi^{2m-1}_Y\times \pi^{2m-1}_Y \times \pi^{2m-1}_Y)_\ast  A^{4m-2}(Y\times Y\times Y)\ \hookrightarrow\   A^{m+1}(C\times C\times C)\ .\]
   But the right-hand side vanishes for dimension reasons for any $m\ge 3$, and so $\Gamma_{ijk}=0$.    
   
\end{proof}

We are now ready to prove the main result of this note:

\begin{theorem}\label{main} Let $Y$ be a Gushel--Mukai fivefold. Then $Y$ has an MCK decomposition.
\end{theorem}

\begin{proof} First, let us treat the case of {\em ordinary\/} Gushel--Mukai fivefolds. We check that ordinary Gushel--Mukai fivefolds $Y$ verify the hypotheses of Proposition \ref{crit}, with
$V$ being the Grassmannian $\Gr(2,5)$ (embedded via the second tensor power of the Pl\"ucker line bundle). Clearly condition (c1) of Proposition \ref{crit} is satisfied. Condition (c2) is Proposition \ref{van}. Condition (c3) is verified, because one has
  \[ \dim A^3(\Gr(2,5))=2 =  \dim A^4(\Gr(2,5))\ ,\]
  and so by hard Lefschetz the composition
   \[ A^3(\Gr(2,5))\ \to\ A^3(Y)\ \to\ A^4(\Gr(2,5)) \]
   is an isomorphism. It follows that ordinary Gushel--Mukai fivefolds $Y$ have a (generically defined) MCK decomposition.

Next, let us extend to {\em all\/} (i.e. both ordinary and special) Gushel--Mukai fivefolds. To this end, let $ \operatorname{CGr}(2,V_5)$ be the cone over the Grassmannian as before. Let us
 write $\YY^\prime\to B^\prime$ for the universal family of all Gushel--Mukai fivefolds, where
  \[ B^\prime\ \subset\ \bar{B}^\prime:= \PP H^0(  \operatorname{CGr}(2,5) ,\OO_{}(1)\oplus \OO_{}(2)  )  \]
  is the dense Zariski open parametrizing smooth dimensionally transverse intersections. 
As in the proof of Proposition \ref{van} above, there is a base change diagram
  \[   \begin{array}[c]{ccc}   \YY^\prime_{} & \to & \YY\\
                                             &&\\
                                             \downarrow&&\downarrow\\
                                             B^\prime_{\rm ord}&\to& \ \ B_{\rm ord} \ . \\
                                             \end{array}\]
                                             
   Thanks to the above first step, we have relative projectors $\pi^j_\YY\in A^{5}(\YY\times_{B_{\rm ord}}\YY)$ which fiberwise give an MCK decomposition.
   Pulling them back via the base change $B^\prime_{\rm ord}\to B_{\rm ord}$ and extending over all of $B^\prime$, we obtain relative projectors $\pi^j_{\YY^\prime}\in A^{5}(\YY^\prime\times_{B^\prime_{}}\YY^\prime)$ which for any fiber over $B^\prime_{\rm ord}$ give an MCK decomposition. The usual spread argument \cite[Lemma 3.2]{Vo} then 
   gives that the restriction to any fiber over $B^\prime$ gives an MCK decomposition, and so we conclude that any Gushel--Mukai fivefold has a (generically defined) MCK decomposition.
     \end{proof}

For later use, let us isolate a consequence of the above proof:

\begin{corollary}\label{Frord} Let $\YY\to B$ be the universal family of ordinary Gushel--Mukai fivefolds. The families $\YY\to B$ and $\YY\times_{B}\YY\to B$ have the Franchetta property.
\end{corollary}

\begin{proof} As we have checked in the proof of Theorem \ref{main}, ordinary Gushel--Mukai fivefolds verify the conditions of Proposition \ref{crit}. Looking at the proof of Proposition \ref{crit}, the result follows.
\end{proof}

\begin{remark}\label{pity} In all likelihood, Gushel--Mukai {\em threefolds\/} also have an MCK decomposition. Unfortunately, our criterion (Proposition \ref{crit}) does not apply because of the dimension condition. For a Gushel--Mukai threefold $Y$, it remains to prove the vanishing
  \[ \pi^3_Y\circ \Delta_Y^{\rm sm}\circ (\pi^3_Y\times \pi^3_Y)\stackrel{??}{=}0\ \ \ \hbox{in}\ A^6(Y^3)\ ,\]
  i.e. one would need an instance of the Franchetta property for $Y^3$. 
\end{remark}

\subsection{Further examples} Here is another family of Fano varieties (of dimension 17) to which our general criterion applies:

\begin{theorem}\label{main2} Let $Y$ be a smooth hyperplane section
  \[  Y:=\Gr(3,9)\cap H  \ \ \subset\ \PP^{83} \]
  (relative to the Pl\"ucker embedding). Then $Y$ has an MCK decomposition.
\end{theorem}

\begin{proof} Let us check that the conditions of Proposition \ref{crit} are verified. Clearly condition (c1) of Proposition \ref{crit} is satisfied. Condition (c2) is \cite[Corollary 4.2]{hyp}, where it is proven that
  \[ A^j_{\rm hom}(Y)=0\ \ \ \forall\ j\not=  9\ .\]
As for condition (c3), we observe that
  \[ \dim A^9(\Gr(3,9))= 8    =  \dim A^{10}(\Gr(3,9))\ \]
  (this is readily checked, using for instance \cite[Theorem 5.26]{3264}).
 By hard Lefschetz, it then follows that the composition
   \[ A^9(\Gr(3,9))\ \to\ A^9(Y)\ \to\ A^{10}(\Gr(3,9)) \]
   is an isomorphism. Since $\Gr(3,9)$ has trivial Chow groups, this implies condition (c3). All conditions of Proposition \ref{crit} are verified, and so the theorem is proven.
\end{proof}

\begin{remark} Varieties $Y$ as in Theorem \ref{main2} are studied in \cite[Section 5.1]{BFM}, where they are related to Coble cubics and abelian surfaces. As shown in loc. cit., the Hodge number $h^{9,8}(Y)=2$, and conjecturally there is a genus 2 curve showing up in the derived category of $Y$.
 \end{remark}

\begin{remark} Other varieties to which Proposition \ref{crit} might {\em perhaps\/} apply are intersections of 3 quadrics in $\PP^{2m+2}$ with $m\ge 3$. Taking $V$ to be the intersection of 2 quadrics in $\PP^{m+2}$, the only problem consists in checking condition (c3); I have not been able to do so.
\end{remark}

 \section{The tautological ring}
 
 \begin{corollary}\label{cor1} Let $Y$ be a Gushel--Mukai fivefold, and $m\in\NN$. Let
  \[ R^\ast(Y^m):=\Bigl\langle (p_i)^\ast A^1(Y), \,  (p_i)^\ast A^2(Y), \, (p_{ij})^\ast(\Delta_Y)\Bigr\rangle\ \subset\ \ \ A^\ast(Y^m)   \]
  be the $\QQ$-subalgebra generated by (pullbacks of) divisors, codimension 2 cycles and (pullbacks of) the diagonal $\Delta_Y\in A^5(Y\times Y)$. (Here $p_i$ and $p_{ij}$ denote the various projections from $Y^m$ to $Y$ resp. to $Y\times Y$).
  The cycle class map induces injections
   \[ R^\ast(Y^m)\ \hookrightarrow\ H^\ast(Y^m,\QQ)\ \ \ \hbox{for\ all\ }m\in\NN\ .\]
   \end{corollary}

\begin{proof} This is inspired by the analogous result for cubic hypersurfaces \cite[Section 2.3]{FLV3}, which in turn is inspired by analogous results for hyperelliptic curves \cite{Ta2}, \cite{Ta} (cf. Remark \ref{tava} below) and for K3 surfaces \cite{Yin}.

As in \cite[Section 2.3]{FLV3}, let us write 
  \[  o:={1\over 10} \, h^5\ \ \in\  A^5(Y)\ , \ \ \ \ c:=c_2(Q)\vert_Y\ \ \in\ A^2(Y)\] 
  (where $Q\to\Gr(2,5)$ is the universal quotient bundle), and
  \[ \tau:= \pi^5_Y=\Delta_Y - \, \sum_{j\not=5}  \pi^j_Y\ \ \in\ A^5(Y\times Y) \ ,\]
  where the $\pi^j_Y$ are as above.

Moreover, for any $1\le i<j\le m$ let us write 
  \[ \begin{split}   o_i&:= (p_i)^\ast(o)\ \ \in\ A^5(Y^m)\ ,\\
                        h_i&:=(p_i)^\ast(h)\ \ \in \ A^1(Y^m)\ ,\\
                        c_i&:=   (p_i)^\ast(c)\ \ \in\ A^2(Y^m)\ ,\\                       
                          \tau_{ij}&:=(p_{ij})^\ast(\tau)\ \ \in\ A^5(Y^m)\ .\\
                         \end{split}\]
 Note that (by definition) we have
   \[ R^\ast(Y^m)= \Bigl\langle    o_i, h_i, c_i, \tau_{ij}\Bigr\rangle\ \ \ \subset\ A^\ast(Y^m)\  .\]                   
                         
  Let us now define the $\QQ$-subalgebra
  \[ \bar{R}^\ast(Y^m):=\Bigl\langle o_i, h_i, c_i  , \tau_{ij}\Bigr\rangle \ \ \ \subset\   H^\ast(Y^m,\QQ) \]
  (where $i$ ranges over $1\le i\le m$, and $1\le i<j\le m$); this is the image of $R^\ast(Y^m)$ in cohomology. One can prove (just as \cite[Lemma 2.12]{FLV3} and \cite[Lemma 2.3]{Yin}) that the $\QQ$-algebra $ \bar{R}^\ast(Y^m)$
  is isomorphic to the free graded $\QQ$-algebra generated by $o_i,h_i,c_i, \tau_{ij}$, modulo the following relations:
    \begin{equation}\label{E:X'}
			\quad h_i \cdot o_i =c_i\cdot o_i= 0,   \quad c_i^3=0,  \quad c_i^2=\lambda\, c_i\cdot h_i^2 =\mu\, h_i^4 ,   \quad 
			h_i^5 =10\,o_i      \,;
			\end{equation}
			\begin{equation}\label{E:X2'}
			\tau_{ij} \cdot o_i =  \tau_{ij} \cdot h_i =\tau_{ij}\cdot c_i= 0, \quad \tau_{ij} \cdot \tau_{ij} =- 20\, o_i\cdot o_j
			\,;
			\end{equation}
			\begin{equation}\label{E:X3'}
			\tau_{ij} \cdot \tau_{ik} = \tau_{jk} \cdot o_i\,;
			\end{equation}
			\begin{equation}\label{E:X4'}
			\sum_{\sigma \in \mathfrak{S}_{22}} 
			\prod_{i=1}^{11} \tau_{\sigma(2i-1), \sigma(2i)} = 0\,. 
			\end{equation}
where $\lambda,\mu\in \QQ$ are certain constants, and $ \mathfrak{S}_{22}$ denotes the symmetric group on 22 elements.
To prove Corollary \ref{cor1}, it suffices to check that all these relations  are verified modulo rational equivalence. In view of the usual spread lemma \cite[Lemma 3.2]{Vo}, it is sufficient to check that this is the case for {\em ordinary\/} Gushel--Mukai fivefolds.

The relations \eqref{E:X'} take place in $R^\ast(Y)$ and so they follow from the Franchetta property for $Y$ (Corollary \ref{Frord}). 
The relations \eqref{E:X2'} take place in $R^\ast(Y^2)$. The first and the last relations are trivially verified, because ($Y$ being Fano) $A^{10}(Y^2)=\QQ$. As for the second relation of \eqref{E:X2'}, this follows from the Franchetta property for $Y\times Y$ (Corollary \ref{Frord}). 
   
 Relation \eqref{E:X3'} takes place in $R^\ast(Y^3)$ and follows from the MCK decomposition. Indeed, we have
   \[  \Delta_Y^{\rm sm}\circ (\pi^5_Y\times\pi^5_Y)=   \pi^{10}_Y\circ \Delta_Y^{\rm sm}\circ (\pi^5_Y\times\pi^5_Y)  \ \ \ \hbox{in}\ A^{10}(Y^3)\ ,\]
   which (using Lieberman's lemma) translates into
   \[ (\pi^5_Y\times \pi^5_Y\times\Delta_Y)_\ast    \Delta_Y^{\rm sm}  =   ( \pi^5_Y\times \pi^5_Y\times\pi^{10}_Y)_\ast \Delta_Y^{\rm sm}                            
                                \ \ \ \hbox{in}\ A^{10}(Y^3)\ ,\]
   which means that
   \[  \tau_{13}\cdot \tau_{23}= \tau_{12}\cdot o_3\ \ \ \hbox{in}\ A^{10}(Y^3)\ .\]
   
  Finally, relation \eqref{E:X4'}, which takes place in $R^\ast(Y^{22})$, can be related to the Kimura finite-dimensionality relation \cite{Kim}:
  relation \eqref{E:X4'} expresses the vanishing
    \[ \sym^{22} H^{5}(Y,\QQ)=0\ ,\]
    where $H^{5}(Y,\QQ)$ is seen as a super vector space.
 This relation is also verified modulo rational equivalence: indeed, relation \eqref{E:X4'} involves a cycle in
   \[ A^\ast(\sym^{22} h^5(Y))\ ,\]
  and $\sym^{22} h^5(Y)$ is $0$ because $Y$ is Kimura finite-dimensional (Proposition \ref{van}). 
   This ends the proof.
%
%
%
 \end{proof}

\begin{remark}\label{tava} Given any curve $C$ and an integer $m\in\NN$, one can define the {\em tautological ring\/}
  \[ R^\ast(C^m):=  \bigl\langle  (p_i)^\ast(K_C),(p_{ij})^\ast(\Delta_C)\bigr\rangle\ \ \ \subset\ A^\ast(C^m) \]
  (where $p_i, p_{ij}$ denote the various projections from $C^m$ to $C$ resp. $C\times C$).
  Tavakol has proven \cite[Corollary 6.4]{Ta} that if $C$ is a hyperelliptic curve, the cycle class map induces injections
    \[  R^\ast(C^m)\ \hookrightarrow\ H^\ast(C^m,\QQ)\ \ \ \hbox{for\ all\ }m\in\NN\ .\]
   On the other hand, there are many (non hyperelliptic) curves for which the tautological ring $R^\ast(C^3)$ does {\em not\/} inject into cohomology (this is related to the non-vanishing of the Ceresa cycle, cf. \cite[Remark 4.2]{Ta} and also \cite[Example 2.3 and Remark 2.4]{FLV2}).   
\end{remark}

%
%
%

 \vskip0.5cm
\begin{nonumberingt} Thanks to Lie Fu and Charles Vial for many inspiring exchanges concerning MCK. Thanks to Len for many beautiful drawings.
\end{nonumberingt}

\end{document}